\documentclass[11pt,reqno]{amsart}

\usepackage[T1]{fontenc}
\usepackage{lmodern}
\usepackage{microtype}
\usepackage{amsmath,amssymb,mathtools,mathrsfs}
\usepackage{enumitem}
\usepackage{booktabs}
\usepackage{xcolor}
\usepackage[colorlinks=true,
            linkcolor=blue!55!black,
            citecolor=blue!55!black,
            urlcolor=blue!55!black]{hyperref}
\hypersetup{
  pdftitle={Minimal equators and homological systoles in Berger projective spaces},
  pdfauthor={Glen Wheeler},
  pdfsubject={Minimal submanifolds and homological systoles in Berger projective spaces},
  pdfkeywords={Berger metric, minimal submanifold, real projective space, Kahler angle, homological systole, integral geometry}
}
\usepackage[nameinlink,noabbrev]{cleveref}

\allowdisplaybreaks
\setlist[enumerate]{leftmargin=*,label=(\roman*)}

\newtheorem{theorem}{Theorem}[section]
\newtheorem{proposition}[theorem]{Proposition}
\newtheorem{lemma}[theorem]{Lemma}
\newtheorem{corollary}[theorem]{Corollary}
\newtheorem{conjecture}{Conjecture}[section]
\theoremstyle{definition}
\newtheorem{definition}[theorem]{Definition}

\theoremstyle{remark}
\newtheorem{remark}[theorem]{Remark}

\crefname{theorem}{theorem}{theorems}
\Crefname{theorem}{Theorem}{Theorems}
\crefname{proposition}{proposition}{propositions}
\Crefname{proposition}{Proposition}{Propositions}
\crefname{lemma}{lemma}{lemmas}
\Crefname{lemma}{Lemma}{Lemmas}
\crefname{corollary}{corollary}{corollaries}
\Crefname{corollary}{Corollary}{Corollaries}
\crefname{conjecture}{conjecture}{conjectures}
\Crefname{conjecture}{Conjecture}{Conjectures}
\crefname{equation}{equation}{equations}
\Crefname{equation}{Equation}{Equations}
\crefname{section}{section}{sections}
\Crefname{section}{Section}{Sections}

\newcommand{\R}{\mathbb R}
\newcommand{\C}{\mathbb C}
\newcommand{\Z}{\mathbb Z}
\newcommand{\Sph}{\mathbb S}
\newcommand{\RP}{\mathbb {RP}}
\newcommand{\CP}{\mathbb {CP}}
\newcommand{\Gr}{\operatorname{Gr}}
\newcommand{\Vol}{\operatorname{Vol}}
\newcommand{\vol}{\operatorname{vol}}
\newcommand{\im}{\operatorname{im}}
\newcommand{\tr}{\operatorname{tr}}
\newcommand{\pr}{\operatorname{pr}}
\newcommand{\sys}{\operatorname{sys}}
\newcommand{\arsinh}{\operatorname{arsinh}}

\newcommand{\dd}{\,\mathrm d}

\newcommand{\Htr}{\mathcal H}
\newcommand{\Vfun}{\mathcal V_\mu}

\newcommand{\ip}[2]{\left\langle #1,#2\right\rangle}
\newcommand{\norm}[1]{\left\lvert #1\right\rvert}

\title[Equators in Berger projective spaces]
{Minimal equators and homological systoles\\
in Berger projective spaces}

\author{Glen Wheeler}
\address{School of Mathematics and Physics, University of Wollongong, Northfields Avenue, Wollongong, NSW 2522, Australia}
\email{glenw@uow.edu.au}
\date{26 July 2026}


\begin{document}


\begin{abstract}
Gil-Medrano asked whether real projective subspaces remain minimal and
volume-minimising under Berger deformations in dimensions greater than
three.  We answer the minimality question and make substantial progress
on the corresponding minimisation problem.  The fundamental observation
is that the full linear geometry of a real subspace
$V\subset\C^N$ is controlled by the single skew-adjoint endomorphism
$A_V=\pr_VJ|_V$.  For every non-round Berger metric, this yields a
complete classification:
$\RP(V)$ is minimal if and only if $V$ is a linear CR subspace,
equivalently if all of its multiple K\"ahler angles are $0$ or $\pi/2$.
Thus the answer to Gil-Medrano's minimality question is negative in
general, although every equatorial hypersurface remains minimal.  The
same structure determines the critical points and global extrema of the
linear volume functional.  Combining it with projective integral
geometry, we also compute exact mod-two homological systoles in the pure
stretched and squashed regimes.
Our final contribution is a conjecture that all remaining cases consist of mixed CR equators.
\end{abstract}

\maketitle

\section{Introduction}

Let $\Sph^{2N-1}=\{p\in\C^N:\norm p=1\}$, where $N\ge2$, carry its round metric $g$.  The Hopf vector field is $\xi(p)=Jp$, where $J$ is multiplication by $\sqrt{-1}$, and the associated one-form is $\eta=g(\xi,\cdot)$.  For $\mu>0$, the Berger metric is
\begin{equation}\label{eq:berger-intro}
 g_\mu=g+(\mu-1)\eta\otimes\eta.
\end{equation}
Thus the horizontal metric is unchanged and the Hopf direction has squared length $\mu$.  The antipodal map is an isometry, and so $g_\mu$ descends to $\RP^{2N-1}$.

For a real $d$-dimensional subspace $V\subset\C^N$, set $M_V=\Sph(V)=V\cap\Sph^{2N-1}$ and $\RP(V)=M_V/\{\pm1\}$.
The submanifold $M_V$ is an equatorial $(d-1)$-sphere for the round metric.  In the round projective space, its quotient $\RP(V)$ is totally geodesic and is a volume minimiser in the non-trivial mod-two homology class.  Berger and Fomenko proved the corresponding global minimisation theorem and its rigidity statement; see \cite{Berger1972,Fomenko1972,Le1993}.

Gil-Medrano asked whether these projective subspaces remain minimal, and whether they remain the only volume-minimising cycles, for the Berger metrics in dimensions greater than three \cite{MorganPansu2018}.  In $\RP^3$, the projections of equatorial two-spheres are minimal but not totally geodesic when $\mu\ne1$, and Gil-Medrano proved that they are precisely the area-minimising projective planes \cite{GilMedrano2016}; see also \cite{BorrelliGilMedrano2010,BrayBrendleEichmairNeves2010}.  Minimal surfaces and higher-dimensional minimal submanifolds in Berger spheres have otherwise been studied from several points of view, including construction, stability and Morse index \cite{Torralbo2012,TorralboUrbano2022}.  A recent survey article of Ambrozio places metrics with minimal equators, their Radon transforms and related projective inverse problems in a broader programme \cite{Ambrozio2026}.  To the best of our knowledge, the higher-codimensional linear question above has not previously been settled.

The point of this paper is to address this by exploiting a new observation: that the full linear geometry is controlled by one skew-adjoint endomorphism.  
Let $\pr_V$ denote round orthogonal projection onto $V$ and define
\begin{equation}\label{eq:A-intro}
 A_V=\pr_VJ|_V\colon V\longrightarrow V.
\end{equation}
Then $A_V\in\mathfrak{so}(V)$, and the singular values of $A_V$ are the cosines of the multiple K\"ahler angles of $V$.  We use the trace convention for mean curvature: $\Htr_\mu=\tr_{g_\mu|_{M_V}}B^\mu$.

Our first result gives the mean curvature and the exact minimality criterion.

\begin{theorem}[Mean curvature and minimality]\label{thm:intro-minimality}
Let $V\subset\C^N$ be a real $d$-plane, let $p\in M_V$, and put $A=A_V$.  Then
\[
 \Htr_\mu(p)
 =\frac{2(\mu-1)}{1+(\mu-1)\norm{Ap}^2}
   \pr_{V^\perp}(JAp).
\]
In particular, if $\mu\ne1$, the following are equivalent:
\begin{enumerate}
 \item $M_V\subset(\Sph^{2N-1},g_\mu)$ is minimal;
 \item $\RP(V)\subset(\RP^{2N-1},g_\mu)$ is minimal;
 \item $J(\im A)\subset V$;
 \item there is an orthogonal splitting $V=C\oplus R$ such that $JC=C$ and $JR\perp V$.
\end{enumerate}
Equivalently, every multiple K\"ahler angle of $V$ belongs to $\{0,\pi/2\}$.
\end{theorem}

The spaces described in \cref{thm:intro-minimality} are the linear CR subspaces of $\C^N$.  The theorem immediately answers the first part of Gil-Medrano's question negatively: a general real subspace has intermediate K\"ahler angles and its equator is not minimal.  There are nevertheless two useful endpoint phenomena.  First, every real hyperplane $V\subset\C^N$ with unit normal $\nu$ has the form $V=\big(\operatorname{span}_{\R}\{\nu,J\nu\}\big)^\perp\oplus\R J\nu$, and therefore every equatorial hypersphere is minimal.  Secondly, a minimal linear equator is totally geodesic precisely when it is pure: either $V$ is complex or $V$ is totally real.  Mixed CR spaces give minimal but non-totally-geodesic examples in every admissible dimension.

The induced volume has an equally direct description.

\begin{theorem}[The linear volume functional]\label{thm:intro-volume}
Let $d=k+1$ and define $\Vfun(V)=\Vol_{g_\mu}(\RP(V))$ for $V\in\Gr_d(\R^{2N})$.
Then
\[
 \Vfun(V)
 =\frac12\int_{\Sph(V)}
   \sqrt{1+(\mu-1)\norm{A_Vp}^2}\,\dd\sigma(p).
\]
For $\mu\ne1$, a real $d$-plane is a critical point of $\Vfun$ if and only if it is a linear CR subspace, equivalently if and only if $\RP(V)$ is minimal.

Set $c_{\min}=\max\{0,d-N\}$ and $c_{\max}=\left\lfloor d/2\right\rfloor$.  If $\mu>1$, the global minimum of $\Vfun$ is attained exactly on the $U(N)$-orbit of $\C^{c_{\min}}\oplus\R^{d-2c_{\min}}$, and the global maximum exactly on the orbit with complex dimension $c_{\max}$.  If $0<\mu<1$, the roles of the two orbits are reversed.
\end{theorem}

For a CR space $V=\C^c\oplus\R^r$, where $d=2c+r$, the volume formula in \cref{thm:intro-volume} evaluates to
\[
 \frac{\Vol_{g_\mu}(\RP(V))}{\Vol_g(\RP^{d-1})}
 ={}_2F_1\left(-\frac12,c;\frac d2;1-\mu\right).
\]
The purely totally real and purely complex cases give the ratios $1$ and $\sqrt\mu$, respectively.  In the mixed cases this is a genuinely non-constant beta integral.

Our partial answer to the global minimisation question is most naturally stated in terms of mod-two homological systoles.  Write $v_k=\Vol_g(\RP^k)=\pi^{(k+1)/2}/\Gamma((k+1)/2)$ and let
\[
 \sys^{\Z_2}_k(g_\mu)
 =\inf\big\{\mathbf M_{g_\mu}(T):
 T\in\mathcal Z_k(\RP^{2N-1};\Z_2),\ [T]\ne0\big\}.
\]

\begin{theorem}[Exact homological systoles]\label{thm:intro-systole}
Let $1\le k\le2N-2$.
\begin{enumerate}
 \item If $\mu>1$ and $k\le N-1$, then $\sys^{\Z_2}_k(g_\mu)=v_k$.  The value is attained by every totally real linear $\RP^k$.
 \item If $0<\mu<1$ and $k$ is odd, then $\sys^{\Z_2}_k(g_\mu)=\sqrt\mu\,v_k$.  The value is attained by every complex linear $\RP^k$.
\end{enumerate}
These are all equality cases.
\end{theorem}

The proof is straightforward once the pointwise Jacobian is calculated.
For a $k$-plane $E\subset T_p\Sph^{2N-1}$,
\[
 \frac{\dd\vol_{g_\mu}|_E}{\dd\vol_g|_E}
 =\sqrt{1+(\mu-1)\norm{\pr_E\xi}^2}.
\]
For stretched fibres this is at least $1$, while for squashed fibres it is at least $\sqrt\mu$.  The remaining step is to apply  the round Crofton inequality: $\mathbf M_g(T)\ge v_k$ for every non-trivial mod-two $k$-cycle in projective space.

The two regimes left open are $\mu>1$ with $k>N-1$, and $0<\mu<1$ with $k$ even.  In both cases the linear minimiser from \cref{thm:intro-volume} is mixed CR and the elementary Jacobian comparison above is not sharp.  We formulate a precise conjecture asserting that these mixed CR projective spaces are globally volume minimising.  Gil-Medrano's theorem in $\RP^3$ is exactly the first non-trivial mixed case.

The paper is organised as follows.  In \cref{sec:preliminaries} we fix the Berger normalisation and record the linear algebra of $A_V$.  In \cref{sec:mean-curvature} we derive the connection and mean-curvature formula.  The CR classification and its geometric consequences are proved in \cref{sec:minimal-equators}.  The volume functional, its critical set and its extrema are treated in \cref{sec:volume}.  In \cref{sec:systoles} we prove the systolic theorem using integral geometry.  The mixed regimes and the main conjecture are discussed in \cref{sec:mixed}.

\section*{Acknowledgements}

The author gratefully acknowledges partial support via  ARC grants 
FT250100880 and DP250101080.

\section{Berger geometry and compressed complex structures}\label{sec:preliminaries}

\subsection{The Berger metric}

We regard $\C^N$ as $\R^{2N}$ with its Euclidean inner product $g=\ip\cdot\cdot$ and orthogonal complex structure $J$.  The same letter $g$ denotes the induced round metric on the unit sphere.  At $p\in\Sph^{2N-1}$, $\xi(p)=Jp$ and $\eta(X)=\ip{Jp}{X}$ for $X\in T_p\Sph^{2N-1}$.
The horizontal space is $\mathcal H_p=\ker\eta_p$.  The metric \eqref{eq:berger-intro} satisfies
\[
 g_\mu|_{\mathcal H_p}=g|_{\mathcal H_p},
 \qquad
 g_\mu(\xi,\xi)=\mu,
 \qquad
 g_\mu(\xi,\mathcal H_p)=0.
\]
Note that we do not normalise the Hopf field to have unit $g_\mu$-length; this convention keeps the deformation parameter equal to the squared vertical length.

For tangent vector fields, define
\begin{equation}\label{eq:phi-def}
 \phi X=\nabla^g_X\xi=JX+\eta(X)p.
\end{equation}
The field $\phi X$ is horizontal and tangent to the sphere.  Moreover, $g(\phi X,Y)=-g(X,\phi Y)$.

\subsection{Linear equators}

Fix a real $d$-plane $V\subset\C^N$, where $1\le d\le2N-1$.  At $p\in M_V$, $T_pM_V=V\cap p^\perp$.  The compression $A=A_V$ from \eqref{eq:A-intro} is skew-adjoint, since $\ip{Ax}{y}=\ip{Jx}{y}=-\ip{x}{Jy}=-\ip{x}{Ay}$ for $x,y\in V$.  It is also a contraction: $\norm{Ax}\le\norm x$.
Since $\ip{Ap}{p}=0$, the round tangential component of the Hopf field along $M_V$ is
\begin{equation}\label{eq:xi-tangent}
 \xi^\top(p)=Ap.
\end{equation}
We write $q(p)=\norm{Ap}^2$.

The skew-adjoint normal form gives an orthonormal basis of $V$ in which
\begin{equation}\label{eq:A-normal-form}
 A\simeq
 \lambda_1J_0\oplus\cdots\oplus\lambda_mJ_0\oplus0_{d-2m},
 \qquad
 J_0=\begin{pmatrix}0&-1\\1&0\end{pmatrix},
\end{equation}
where $m=\lfloor d/2\rfloor$ and $1\ge\lambda_1\ge\cdots\ge\lambda_m\ge0$.  If $d$ is even there is no final zero block; if $d$ is odd it has dimension one.  The numbers $\theta_j=\arccos\lambda_j\in[0,\pi/2]$ are the multiple K\"ahler angles of $V$.  They determine the $U(N)$-orbit of $V$; see Tasaki \cite{Tasaki2003}.  In the coordinates of \eqref{eq:A-normal-form},
\begin{equation}\label{eq:q-coordinates}
 q(p)=\sum_{j=1}^m\lambda_j^2
 \big(x_{2j-1}^2+x_{2j}^2\big),
 \qquad p=(x_1,\ldots,x_d)\in\Sph^{d-1}.
\end{equation}

The endpoint singular values have a direct geometric meaning.

\begin{lemma}\label{lem:endpoint-angles}
Let $A=A_V$.
\begin{enumerate}
 \item The $\lambda_j=1$ blocks form the maximal complex subspace of $V$, namely $V\cap JV$.
 \item $A=0$ if and only if $V$ is totally real, namely $JV\perp V$.
 \item All singular values of $A$ belong to $\{0,1\}$ if and only if $V$ has an orthogonal splitting $V=C\oplus R$ with $JC=C$ and $JR\perp V$.
\end{enumerate}
\end{lemma}

\begin{proof}
The intersection $V\cap JV$ is $J$-invariant.  On this intersection, $A=J$, so it is contained in the sum of the unit singular-value blocks.  Conversely, if $\norm{Ax}=\norm x$, equality holds in the orthogonal projection estimate $\norm{Ax}\le\norm x$; hence $Jx\in V$.  This proves the first assertion.  The second is immediate from the definition of $A$.

If the singular values are in $\{0,1\}$, let $C$ be the sum of the unit blocks and $R=\ker A$.  The first assertion gives $JC=C$, while $Ar=0$ is equivalent to $Jr\perp V$.  The converse follows by restricting $A$ to $C$ and $R$.
\end{proof}

The complex dimension of $V\cap JV$ is constrained by the ambient dimension.

\begin{lemma}\label{lem:intersection-bound}
If $\dim_\R V=d$, then $\dim_\R(V\cap JV)\ge2(d-N)$.  Consequently, at least $c_{\min}=\max\{0,d-N\}$ of the blocks in \eqref{eq:A-normal-form} have $\lambda_j=1$.
\end{lemma}

\begin{proof}
Both $V$ and $JV$ have real dimension $d$ in a space of real dimension $2N$, and therefore $\dim(V\cap JV)=\dim V+\dim JV-\dim(V+JV)\ge2d-2N$.
The intersection is complex, and \cref{lem:endpoint-angles} identifies it with the unit singular-value blocks.
\end{proof}

\section{Connection and mean curvature}\label{sec:mean-curvature}

We first relate the two Levi-Civita connections.  A version of this formula, in the unit vertical-field convention, appears in \cite[Lemma 1]{TorralboUrbano2022}.

\begin{lemma}[Connection formula]\label{lem:connection}
Let $\nabla^\mu$ and $\nabla^g$ denote the Levi-Civita connections of $g_\mu$ and $g$, respectively.  For tangent vector fields $X,Y$ on $\Sph^{2N-1}$,
\[
 \nabla^\mu_XY
 =\nabla^g_XY
 +(\mu-1)\big\{\eta(Y)\phi X+\eta(X)\phi Y\big\}.
\]
\end{lemma}

\begin{proof}
Put $a=\mu-1$ and let $D=\nabla^\mu-\nabla^g$.  Since $g_\mu=g+a\eta\otimes\eta$ and $\nabla^gg=0$, the difference-of-connections formula gives
\begin{align*}
 2g_\mu(D_XY,Z)
 &=a\big\{(\nabla^g_X\eta)(Y)\eta(Z)
       +\eta(Y)(\nabla^g_X\eta)(Z)\\
 &\qquad +(\nabla^g_Y\eta)(X)\eta(Z)
       +\eta(X)(\nabla^g_Y\eta)(Z)\\
 &\qquad -(\nabla^g_Z\eta)(X)\eta(Y)
       -\eta(X)(\nabla^g_Z\eta)(Y)\big\}.
\end{align*}
By \eqref{eq:phi-def}, $(\nabla^g_X\eta)(Y)=g(\phi X,Y)$.  Using the skew-adjointness of $\phi$, the terms proportional to $\eta(Z)$ cancel and the remaining terms reduce to $g_\mu(D_XY,Z)=a\{\eta(Y)g(\phi X,Z)+\eta(X)g(\phi Y,Z)\}$.  As $\phi X$ and $\phi Y$ are horizontal, their $g$- and $g_\mu$-pairings agree.  This proves the asserted connection formula.
\end{proof}

Let $B^\mu$ be the second fundamental form of $M_V$ in $(\Sph^{2N-1},g_\mu)$.  Because $M_V$ is totally geodesic for the round metric, \cref{lem:connection} gives
\begin{equation}\label{eq:B-pretrace}
 B^\mu(X,Y)
 =(\mu-1)
 \big(\eta(Y)\phi X+\eta(X)\phi Y\big)^{\perp_\mu},
 \qquad X,Y\in TM_V.
\end{equation}
The superscript $\perp_\mu$ denotes $g_\mu$-orthogonal projection onto the normal bundle.

\begin{proposition}[Mean-curvature formula]\label{prop:mean-curvature}
For $p\in M_V$,
\[
 \Htr_\mu(p)
 =\frac{2(\mu-1)}{1+(\mu-1)\norm{Ap}^2}
  \pr_{V^\perp}(JAp).
\]
The vector $\pr_{V^\perp}(JAp)$ is horizontal and hence has the same $g$- and $g_\mu$-norm.
\end{proposition}

\begin{proof}
On $T_pM_V$, let $\alpha=\eta|_{T_pM_V}$.  By \eqref{eq:xi-tangent}, its round dual is $\alpha^{\sharp_g}=Ap$, while the induced Berger metric is the rank-one perturbation $h_\mu=g|_{T_pM_V}+(\mu-1)\alpha\otimes\alpha$.  The Sherman--Morrison formula therefore gives $\alpha^{\sharp_{h_\mu}}=Ap/(1+(\mu-1)\norm{Ap}^2)$.
Tracing \eqref{eq:B-pretrace} with respect to $h_\mu$ and using linearity of $\phi$ yields
\[
 \Htr_\mu(p)
 =\frac{2(\mu-1)}{1+(\mu-1)\norm{Ap}^2}
   \big(\phi(Ap)\big)^{\perp_\mu}.
\]
Now $\eta(Ap)=\ip{Jp}{Ap}=\norm{Ap}^2$, so $\phi(Ap)=JAp+\norm{Ap}^2p$.  Since $\pr_V(JAp)=A(Ap)=A^2p$, we can decompose this as
\begin{equation}\label{eq:phi-decomposition}
 \phi(Ap)
 =\underbrace{\big(A^2p+\norm{Ap}^2p\big)}_{\in T_pM_V}
 +\underbrace{\pr_{V^\perp}(JAp)}_{=:L_p}.
\end{equation}
The first vector is tangent because $\ip{A^2p+\norm{Ap}^2p}{p}=-\norm{Ap}^2+\norm{Ap}^2=0$.
The normal vector $L_p$ is horizontal.  Indeed,
\begin{align*}
 \ip{L_p}{Jp}
 =\ip{JAp-A^2p}{Jp}
 =\ip{Ap}{p}-\ip{A^2p}{Ap}=0,
\end{align*}
where the last equality uses the skew-adjointness of $A$.  Thus $L_p$ is also $g_\mu$-normal to $M_V$, and the preceding expression gives the asserted mean-curvature formula.
\end{proof}

The norm of the mean curvature exposes the intermediate K\"ahler angles directly.

\begin{corollary}\label{cor:H-norm}
At $p\in M_V$,
\[
 \norm{\Htr_\mu(p)}^2_{g_\mu}
 =\frac{4(\mu-1)^2}
 {\big(1+(\mu-1)\norm{Ap}^2\big)^2}
 \Big(\norm{Ap}^2-\norm{A^2p}^2\Big).
\]
In the coordinates of \eqref{eq:A-normal-form},
\[
 \norm{Ap}^2-\norm{A^2p}^2
 =\sum_{j=1}^m\lambda_j^2(1-\lambda_j^2)
 \big(x_{2j-1}^2+x_{2j}^2\big).
\]
\end{corollary}

\begin{proof}
The orthogonal decomposition in \eqref{eq:phi-decomposition} gives $\norm{\pr_{V^\perp}(JAp)}^2=\norm{JAp}^2-\norm{A^2p}^2=\norm{Ap}^2-\norm{A^2p}^2$.  The vector is horizontal, so the same identity holds for $g_\mu$.  The second formula follows from the block normal form.
\end{proof}

\section{Minimal linear equators}\label{sec:minimal-equators}

We now prove the classification in \cref{thm:intro-minimality}.

\begin{theorem}[CR classification]\label{thm:CR-classification}
Assume $\mu\ne1$.  Let $V\subset\C^N$ be a real subspace and $A=A_V$.  The following are equivalent:
\begin{enumerate}
 \item $M_V$ is minimal in $(\Sph^{2N-1},g_\mu)$;
 \item $\RP(V)$ is minimal in $(\RP^{2N-1},g_\mu)$;
 \item $\pr_{V^\perp}JA=0$;
 \item $J(\im A)\subset V$;
 \item $V$ is a linear CR subspace, namely $V=C\oplus R$ with $JC=C$ and $JR\perp V$;
 \item all multiple K\"ahler angles of $V$ lie in $\{0,\pi/2\}$.
\end{enumerate}
\end{theorem}

\begin{proof}
The quotient map $\Sph^{2N-1}\to\RP^{2N-1}$ is a local isometry, so the first two assertions are equivalent.  By \cref{prop:mean-curvature}, minimality is equivalent to $\pr_{V^\perp}(JAp)=0$ for every $p\in\Sph(V)$.  By homogeneity this is equivalent to $\pr_{V^\perp}JA=0$, or $J(\im A)\subset V$.

Suppose $J(\im A)\subset V$.  Set $C=\im A$ and $R=\ker A$.  Since $A$ is skew-adjoint, $V=C\oplus R$ orthogonally.  If $y\in C$, then $Jy\in V$ by hypothesis, and hence $Ay=\pr_VJy=Jy$.  In particular $Jy\in C$, so $JC=C$.  If $r\in R$, then $0=Ar=\pr_VJr$, which is equivalent to $Jr\perp V$.  Thus $V$ is CR.  The converse is immediate: on $C$ one has $A=J$, while on $R$ one has $A=0$.  The final equivalence follows from \cref{lem:endpoint-angles}.
\end{proof}

\subsection{A negative answer to Gil-Medrano's question.}

The following two-plane gives a minimality obstruction in every ambient dimension $N\ge2$.  Choose orthonormal vectors $e_1,e_2\in\C^N$ such that $e_1,Je_1,e_2,Je_2$ are mutually orthonormal, and let
\[
 V_\theta
 =\operatorname{span}_{\R}
 \big\{e_1,\cos\theta\,Je_1+\sin\theta\,e_2\big\},
 \qquad0<\theta<\frac\pi2.
\]
Its unique K\"ahler-angle cosine is $\lambda=\cos\theta$.  Thus it is neither complex nor totally real and is not minimal for $\mu\ne1$.  More explicitly, at $p=e_1$, \cref{prop:mean-curvature} gives
\[
 \Htr_\mu(e_1)
 =\frac{2(\mu-1)\cos\theta\sin\theta}
 {1+(\mu-1)\cos^2\theta}\,Je_2.
\]
This disproves the assertion that all equatorial projective subspaces remain minimal.

\subsection{Hyperplanes and total geodesy}

\begin{corollary}[Equatorial hyperspheres]\label{cor:hyperplanes}
Every real hyperplane $V\subset\C^N$ defines a minimal equatorial hypersphere in every Berger sphere, and hence a minimal projective hyperplane in every Berger projective space.
\end{corollary}

\begin{proof}
Let $\nu$ be a unit normal to $V$.  Then $V=C\oplus R$, where $C=\operatorname{span}_{\R}\{\nu,J\nu\}^\perp$ and $R=\R J\nu$.  The subspace $C$ is complex and $JR=\R\nu\perp V$.  Apply \cref{thm:CR-classification}.
\end{proof}

The mixed examples are generally not totally geodesic.

\begin{proposition}[Totally geodesic linear equators]\label{prop:totally-geodesic}
Assume $\mu\ne1$ and let $M_V$ be minimal.  Then $M_V$ is totally geodesic in $(\Sph^{2N-1},g_\mu)$ if and only if either
\begin{enumerate}
 \item $V$ is complex, or
 \item $V$ is totally real.
\end{enumerate}
The same statement holds after passage to projective space.
\end{proposition}

\begin{proof}
Write $V=C\oplus R$ as in \cref{thm:CR-classification}.  If $R=0$, then $V$ is complex.  For tangent $X,Y$, both $\nabla^g_XY$ and the correction in the formula of \cref{lem:connection} lie in $V\cap p^\perp$, so the second fundamental form vanishes.  If $C=0$, then $\eta|_{TM_V}=0$, and \cref{lem:connection} again shows that the second fundamental form vanishes.

Suppose now that $C$ and $R$ are both non-zero.  Choose $p\in\Sph(C)$, set $X=Jp\in T_pM_V$, and take a non-zero $Y\in R$.  Then $\eta(X)=1$, $\eta(Y)=0$ and $\phi Y=JY$.  Moreover $JY\perp V$, hence is $g_\mu$-normal.  Formula \eqref{eq:B-pretrace} gives $B^\mu(X,Y)=(\mu-1)JY\ne0$.
Thus a mixed CR equator is not totally geodesic.  The quotient statement follows from local isometry.
\end{proof}

This agrees with the general classification of totally geodesic submanifolds in Berger spheres obtained by Torralbo and Urbano \cite[Proposition 5]{TorralboUrbano2022}.

\subsection{The moduli of minimal linear equators}

Let $d=2c+r$.  A CR subspace of type $(c,r)$ is unitarily congruent to $\C^c\oplus\R^r\subset\C^c\oplus\C^r\oplus\C^{N-c-r}$, where the real factor is the standard real form in $\C^r$.  Such a subspace exists if and only if $c,r\ge0$, $c+r\le N$ and $2c+r=d$.  Equivalently, $c_{\min}=\max\{0,d-N\}\le c\le c_{\max}=\left\lfloor d/2\right\rfloor$.

\begin{proposition}[Critical orbits]\label{prop:critical-orbits}
For every admissible pair $(c,r)$, the $U(N)$-orbit of CR $d$-planes of type $(c,r)$ is the homogeneous space
\[
 \mathcal M_{c,r}
 \cong
 \frac{U(N)}{U(c)\times O(r)\times U(N-c-r)}.
\]
For $\mu\ne1$, the set of minimal linear equators is the disjoint union of these orbits over the admissible range above.
\end{proposition}

\begin{proof}
The unitary group acts transitively on the choices of a complex $c$-plane and, in its Hermitian orthogonal complement, a totally real $r$-plane.  The stabiliser of the standard model is $U(c)\times O(r)\times U(N-c-r)$.  The final statement is \cref{thm:CR-classification}.
\end{proof}

\section{The linear volume functional}\label{sec:volume}

\subsection{Volume and K\"ahler angles}

The rank-one structure of the induced metric gives an exact formula.

\begin{proposition}[Volume formula]\label{prop:volume-formula}
Let $V\subset\C^N$ be a real $d$-plane.  On $M_V$, $\dd\vol_{g_\mu|_{M_V}}=\sqrt{1+(\mu-1)\norm{A_Vp}^2}\,\dd\sigma_g(p)$.  Consequently,
\[
 \Vol_{g_\mu}(\RP(V))
 =\frac12\int_{\Sph(V)}
 \sqrt{1+(\mu-1)\norm{A_Vp}^2}\,\dd\sigma_g(p).
\]
\end{proposition}

\begin{proof}
On $T_pM_V$, the induced metric is $h_\mu=g+(\mu-1)\alpha\otimes\alpha$, where $\alpha=\eta|_{T_pM_V}$.  The matrix determinant lemma and \eqref{eq:xi-tangent} give $\det h_\mu/\det g=1+(\mu-1)\norm{\alpha^{\sharp_g}}^2=1+(\mu-1)\norm{A_Vp}^2$.  Taking square roots proves the first formula in \cref{prop:volume-formula}; the antipodal quotient is two-to-one and isometric, giving the second.
\end{proof}

In the normal form \eqref{eq:A-normal-form}, the volume is a symmetric function of $\lambda_1^2,\ldots,\lambda_m^2$.  It is strictly monotone in each variable.

\begin{lemma}[Angle monotonicity]\label{lem:angle-monotonicity}
Fix all K\"ahler-angle cosines except $\lambda_j$.  Then $\Vol_{g_\mu}(\RP(V))$ is strictly increasing in $\lambda_j^2$ if $\mu>1$, and strictly decreasing in $\lambda_j^2$ if $0<\mu<1$.
\end{lemma}

\begin{proof}
By \eqref{eq:q-coordinates} and \cref{prop:volume-formula}, differentiation under the integral gives
\begin{align*}
 \frac{\partial}{\partial(\lambda_j^2)}
 \Vol_{g_\mu}(\RP(V))
 =\frac{\mu-1}{4}
 \int_{\Sph^{d-1}}
 \frac{x_{2j-1}^2+x_{2j}^2}
 {\sqrt{1+(\mu-1)q(x)}}\,\dd\sigma(x).
\end{align*}
The integral is positive, and the sign is that of $\mu-1$.
\end{proof}

\subsection{Global extrema among linear projective spaces}

\begin{theorem}[Linear volume extrema]\label{thm:linear-extrema}
Let $1\le d\le2N-1$, set $c_{\min}$ and $c_{\max}$ as in \cref{thm:intro-volume}, and let $\mu\ne1$.
\begin{enumerate}
 \item If $\mu>1$, the minimum of $\Vfun$ on $\Gr_d(\R^{2N})$ is attained exactly on the orbit $V\cong\C^{c_{\min}}\oplus\R^{d-2c_{\min}}$, and the maximum exactly on the orbit $V\cong\C^{c_{\max}}\oplus\R^{d-2c_{\max}}$.
 \item If $0<\mu<1$, the two orbits are reversed.
\end{enumerate}
In particular, every linear volume extremiser is minimal.
\end{theorem}

\begin{proof}
By \cref{lem:intersection-bound}, at least $c_{\min}$ of the K\"ahler singular values are forced to equal $1$.  Every other singular value lies in $[0,1]$.  For $\mu>1$, \cref{lem:angle-monotonicity} shows that the volume is minimised by taking exactly the forced $c_{\min}$ singular values equal to $1$ and all remaining singular values equal to $0$.  This is the first CR model in the statement.  The volume is maximised by taking every available singular value equal to $1$, giving the second model.  Both patterns are admissible: in the first case the real dimension of the totally real factor is $d-2c_{\min}$ and $c_{\min}+d-2c_{\min}\le N$; in the second case the real factor has dimension $0$ or $1$.

Strict angle monotonicity identifies the equality patterns.  The multiple K\"ahler angles classify the $U(N)$-orbits, so these are the only equality orbits.  When $0<\mu<1$, all monotonicities reverse.  The final assertion follows from \cref{thm:CR-classification}.
\end{proof}

\subsection{Critical points and minimality}

A general family of submanifolds can be minimal without being detected by a restricted family of variations.  For linear equators, however, one carefully chosen linear variation detects every failure of minimality.

\begin{theorem}[Critical set of the linear volume]\label{thm:critical-set}
Assume $\mu\ne1$.  A real $d$-plane $V$ is a critical point of $\Vfun\colon\Gr_d(\R^{2N})\longrightarrow\R$ if and only if $M_V$ is minimal, equivalently if and only if $V$ is a linear CR subspace.
\end{theorem}

\begin{proof}
Minimality implies stationarity under every variation, hence under variations through linear equators.  It remains to prove the converse.

Let $L=\pr_{V^\perp}JA\colon V\longrightarrow V^\perp$.  If $L\ne0$, define a skew-adjoint endomorphism of $\C^N=V\oplus V^\perp$ by
\[
 K=\begin{pmatrix}0&-L^*\\L&0\end{pmatrix}
\]
and set $V_t=e^{tK}V$.  The corresponding variation of $M_V$ has velocity $X(p)=Kp=Lp$.
As shown in the proof of \cref{prop:mean-curvature}, $Lp$ is horizontal and $g_\mu$-normal.  The first variation formula, together with the two-to-one quotient, gives
\begin{align}
 \left.\frac{\dd}{\dd t}\right|_{t=0}\Vfun(V_t)
 &=-\frac12\int_{M_V}
   g_\mu(\Htr_\mu,Lp)\,\dd\vol_{g_\mu}\notag\\
 &=-(\mu-1)
   \int_{M_V}
   \frac{\norm{Lp}^2}
   {\sqrt{1+(\mu-1)\norm{Ap}^2}}\,\dd\sigma_g.
 \label{eq:critical-variation}
\end{align}
The denominator is positive.  Since $\mu\ne1$, the derivative is non-zero whenever $L\ne0$.  Thus every critical point has $L=0$, which is the minimality condition in \cref{thm:CR-classification}.
\end{proof}

\begin{remark}\label{rem:gradient-direction}
Formula \eqref{eq:critical-variation} gives more than criticality.  For stretched fibres, the variation generated by $L$ strictly decreases volume; for squashed fibres, its negative does.  Thus $L=\pr_{V^\perp}JA$ is a canonical descent direction away from the CR locus.
\end{remark}

\subsection{Closed formulae for CR equators}

Let $V=\C^c\oplus\R^r$ with $d=2c+r$.  For $p\in\Sph(V)$, let $s$ be the squared norm of its projection to the complex factor.  Then $q(p)=s$.  Under normalised spherical measure, $s$ has the beta distribution with parameters $(c,r/2)$ when $c,r>0$.

\begin{proposition}[Hypergeometric volume]\label{prop:hypergeometric-volume}
Let $V=\C^c\oplus\R^r$ be a CR subspace and $d=2c+r$.  Then
\[
 \frac{\Vol_{g_\mu}(\RP(V))}{\Vol_g(\RP^{d-1})}
 ={}_2F_1\left(-\frac12,c;\frac d2;1-\mu\right).
\]
If $c,r>0$, equivalently
\[
 \frac{\Vol_{g_\mu}(\RP(V))}{\Vol_g(\RP^{d-1})}
 =\frac1{B(c,r/2)}
 \int_0^1
 \sqrt{1+(\mu-1)s}\,
 s^{c-1}(1-s)^{r/2-1}\,\dd s.
\]
For $c=0$ the ratio is $1$, and for $r=0$ it is $\sqrt\mu$.
\end{proposition}

\begin{proof}
The beta-integral formula follows from \cref{prop:volume-formula} and the beta distribution of the squared length of the projection of a uniformly distributed point on $\Sph^{d-1}$ to a $2c$-dimensional subspace.  Euler's integral representation of the Gauss hypergeometric function gives the hypergeometric formula, since $1+(\mu-1)s=1-(1-\mu)s$.  The endpoint cases follow directly from $q\equiv0$ and $q\equiv1$.
\end{proof}

The formula also yields a useful expansion near the round metric.

\begin{proposition}[Near-round expansion]\label{prop:near-round}
Let $B=A^*A=-A^2$ and put $a=\mu-1$.  As $a\to0$,
\begin{align*}
 \frac{\Vol_{g_\mu}(\RP(V))}{\Vol_g(\RP^{d-1})}
 &=1+\frac{a}{2d}\tr B\\
 &\quad-\frac{a^2}{8d(d+2)}
 \big\{(\tr B)^2+2\tr(B^2)\big\}
 +O(a^3).
\end{align*}
The error is uniform on the Grassmannian.
\end{proposition}

\begin{proof}
Expand $\sqrt{1+aq}=1+\frac a2q-\frac{a^2}{8}q^2+O(a^3)$.  For $p$ uniformly distributed on $\Sph^{d-1}$, $\mathbb E\,\ip{Bp}{p}=\tr B/d$ and $\mathbb E\,\ip{Bp}{p}^2=((\tr B)^2+2\tr(B^2))/(d(d+2))$.  Insert these identities into the integral formula of \cref{prop:volume-formula}.  Uniformity follows from $0\le B\le I$.
\end{proof}

\subsection{The three-dimensional case}

When $N=2$ and $d=3$, every real hyperplane has CR type $(c,r)=(1,1)$, and all are congruent under $U(2)$.  Since $\Vol_g(\RP^2)=2\pi$, \cref{prop:hypergeometric-volume} gives $\Vol_{g_\mu}(\RP(V))=2\pi I_\mu$, where $I_\mu=\int_0^1\sqrt{\mu+(1-\mu)t^2}\,\dd t$.  The elementary antiderivative yields
\[
 I_\mu=
 \begin{cases}
 \displaystyle
 \frac12+
 \frac{\mu}{2\sqrt{1-\mu}}
 \arsinh\sqrt{\frac{1-\mu}{\mu}},
 &0<\mu<1,\\[2ex]
 1,&\mu=1,\\[1ex]
 \displaystyle
 \frac12+
 \frac{\mu}{2\sqrt{\mu-1}}
 \arcsin\sqrt{\frac{\mu-1}{\mu}},
 &\mu>1.
 \end{cases}
\]
Gil-Medrano proved that these mixed CR planes are not merely linear extrema: they minimise area globally in the non-trivial class of $H_2(\RP^3;\Z_2)$ for every $\mu>0$ \cite{GilMedrano2016}.

\section{Homological systoles}\label{sec:systoles}

We now pass from linear projective subspaces to arbitrary cycles.  Throughout this section, currents have coefficients in $\Z_2$.  This is the natural coefficient group because $H_k(\RP^{2N-1};\Z_2)\cong\Z_2$ for $0\le k\le2N-1$.
We use the standard mass notation $\mathbf M_h$ for a metric $h$; see Federer \cite{Federer1969}.

\subsection{The round projective inequality}

We recall the Crofton form of the Berger--Fomenko theorem.  Let $m=2N-1$ for the moment and let $T$ be a rectifiable mod-two $k$-cycle in $\RP^m$.  For a complementary real projective subspace $\RP(W)\cong\RP^{m-k}$ in general position, the slice $T\cap\RP(W)$ is zero-dimensional.  There is an $O(m+1)$-invariant probability measure $\nu$ on the space of such $W$ for which the Crofton formula reads
\begin{equation}\label{eq:crofton}
 \int \#\big(T\cap\RP(W)\big)\,\dd\nu(W)
 =\frac{\mathbf M_g(T)}{v_k}.
\end{equation}
The normalisation follows by taking $T$ to be a linear $\RP^k$, which meets a generic complementary projective space in one point.  Formula \eqref{eq:crofton} is a special case of the kinematic formula in homogeneous spaces; see Howard \cite{Howard1993} and L\^e \cite{Le1993}.  The rectifiable-current formulation follows by slicing and approximation.

If $[T]\ne0$, then the mod-two intersection number with the complementary generator is one.  Hence a generic transverse slice contains an odd, and therefore non-zero, number of points.  It follows from \eqref{eq:crofton} that
\begin{equation}\label{eq:round-lower-bound}
 \mathbf M_g(T)\ge v_k.
\end{equation}
This is the only part of the classical theorem needed for the value of the systoles below.  The equality statement of Berger and Fomenko says that equality forces $T$ to be a linear projective $k$-plane in the appropriate category \cite{Berger1972,Fomenko1972}.

\subsection{The Berger Jacobian}

\begin{lemma}[Pointwise Jacobian]\label{lem:pointwise-jacobian}
Let $E\subset T_p\Sph^{2N-1}$ be a $k$-plane.  Then
\[
 J_E(g_\mu,g)
 :=\frac{\dd\vol_{g_\mu}|_E}{\dd\vol_g|_E}
 =\sqrt{1+(\mu-1)\norm{\pr_E\xi}^2}.
\]
Consequently,
\[
 J_E(g_\mu,g)\ge
 \begin{cases}
  1,&\mu\ge1,\\
  \sqrt\mu,&0<\mu\le1.
 \end{cases}
\]
Equality in the first line with $\mu>1$ holds if and only if $\xi\perp E$, while equality in the second line with $\mu<1$ holds if and only if $\xi\in E$.
\end{lemma}

\begin{proof}
On $E$, $g_\mu|_E=g|_E+(\mu-1)\eta|_E\otimes\eta|_E$.  The matrix determinant lemma gives $J_E(g_\mu,g)^2=1+(\mu-1)\norm{(\eta|_E)^{\sharp_g}}^2=1+(\mu-1)\norm{\pr_E\xi}^2$.
Since $0\le\norm{\pr_E\xi}^2\le1$, the bounds and equality cases follow.
\end{proof}

The same formula holds on projective space because the quotient is a local isometry.  Integrating over approximate tangent planes gives, for every rectifiable $k$-current $T$,
\begin{equation}\label{eq:mass-comparison}
 \mathbf M_{g_\mu}(T)\ge
 \begin{cases}
  \mathbf M_g(T),&\mu\ge1,\\
  \sqrt\mu\,\mathbf M_g(T),&0<\mu\le1.
 \end{cases}
\end{equation}

\subsection{Exact systoles in the pure regimes}

\begin{definition}\label{def:systole}
For $1\le k\le2N-2$, define
\[
 \sys_k^{\Z_2}(g_\mu)
 =\inf\big\{\mathbf M_{g_\mu}(T):
 T\in\mathcal Z_k(\RP^{2N-1};\Z_2),\ [T]\ne0\big\}.
\]
\end{definition}

\begin{theorem}[Stretched Hopf fibres]\label{thm:stretched-systole}
Suppose $\mu>1$ and $1\le k\le N-1$.  Then $\sys_k^{\Z_2}(g_\mu)=v_k$, and every totally real linear projective $k$-plane attains equality.
\end{theorem}

\begin{proof}
If $[T]\ne0$, then \eqref{eq:mass-comparison} and \eqref{eq:round-lower-bound} give $\mathbf M_{g_\mu}(T)\ge\mathbf M_g(T)\ge v_k$.
Because $k+1\le N$, there exists a totally real $(k+1)$-plane $V\subset\C^N$.  Along $\RP(V)$ the Hopf field is normal, so $g_\mu$ induces the round metric and $\Vol_{g_\mu}(\RP(V))=v_k$.
The cycle $\RP(V)$ represents the non-trivial mod-two homology class, proving equality.
\end{proof}

\begin{theorem}[Squashed Hopf fibres]\label{thm:squashed-systole}
Suppose $0<\mu<1$ and $k$ is odd.  Then $\sys_k^{\Z_2}(g_\mu)=\sqrt\mu\,v_k$, and every complex linear projective $k$-plane attains equality.
\end{theorem}

\begin{proof}
For every non-trivial $T$, \eqref{eq:mass-comparison} and \eqref{eq:round-lower-bound} give $\mathbf M_{g_\mu}(T)\ge\sqrt\mu\,\mathbf M_g(T)\ge\sqrt\mu\,v_k$.
Since $k$ is odd, $k+1$ is even and there exists a complex subspace $V\subset\C^N$ of real dimension $k+1$.  The Hopf field is tangent to $\RP(V)$.  By \cref{lem:pointwise-jacobian}, or \cref{prop:hypergeometric-volume}, $\Vol_{g_\mu}(\RP(V))=\sqrt\mu\,v_k$.
Again $\RP(V)$ represents the non-trivial class.
\end{proof}

\begin{corollary}[Rigidity, using Berger--Fomenko]\label{cor:systole-rigidity}
Assume the equality statement in the classical round projective minimisation theorem.
\begin{enumerate}
 \item In \cref{thm:stretched-systole}, every minimiser is a totally real linear $\RP^k$, modulo $U(N)$.
 \item In \cref{thm:squashed-systole}, every minimiser is a complex linear $\RP^k$, modulo $U(N)$.
\end{enumerate}
\end{corollary}

\begin{proof}
In the stretched case, equality in $\mathbf M_{g_\mu}(T)\ge\mathbf M_g(T)\ge v_k$ forces $T$ to be a round linear projective plane and its approximate tangent planes to be horizontal almost everywhere.  If $T=\RP(V)$ is linear, horizontality means $Jp\perp V\cap p^\perp$ for all $p\in\Sph(V)$.  Since $Jp\perp p$, this is equivalent to $Jp\perp V$ for all $p$, hence $V$ is totally real.

In the squashed case, equality forces a round linear projective plane whose tangent spaces contain the Hopf direction.  Thus $Jp\in V$ for every $p\in V$, so $V$ is complex.
\end{proof}

\section{The remaining mixed regimes}\label{sec:mixed}

Theorems \ref{thm:stretched-systole} and \ref{thm:squashed-systole} cover exactly the cases in which the linear minimiser from \cref{thm:linear-extrema} is pure.  We now describe the complementary cases.

Let $d=k+1$.  For stretched fibres, the linear minimum has type $c=d-N$ and $r=2N-d$ whenever $d>N$.  Both $c$ and $r$ are positive because $d\le2N-1$.  Thus the candidate is mixed CR.  Its volume is
\[
 v_k\,{}_2F_1\left(
 -\frac12,k+1-N;\frac{k+1}{2};1-\mu
 \right).
\]

For squashed fibres and even $k$, the linear minimum has type $c=k/2$ and $r=1$, and volume
\[
 v_k\,{}_2F_1\left(
 -\frac12,\frac k2;\frac{k+1}{2};1-\mu
 \right).
\]
The pointwise estimates in \cref{lem:pointwise-jacobian} are not sharp on these mixed tangent spaces, so a new global inequality is required.

\begin{conjecture}[Mixed CR minimisers]\label{conj:mixed}
Let $1\le k\le2N-2$, $d=k+1$, and $\mu>0$.
\[
 \sys_k^{\Z_2}(g_\mu)
 =\min_{V\in\Gr_d(\R^{2N})}
   \Vol_{g_\mu}(\RP(V)).
\]
If $\mu\ne1$, moreover, every minimiser is, modulo $U(N)$,
\[
 V\cong
 \begin{cases}
 \C^{c_{\min}}\oplus\R^{d-2c_{\min}},&\mu>1,\\
 \C^{c_{\max}}\oplus\R^{d-2c_{\max}},&0<\mu<1,
 \end{cases}
\]
where $c_{\min}$ and $c_{\max}$ are as in \cref{thm:intro-volume}.  When $\mu=1$, every round linear projective $k$-plane is a minimiser, and the equality family is the full $O(2N)$-orbit.
\end{conjecture}

The conjecture is proved by \cref{thm:stretched-systole,thm:squashed-systole} whenever the minimiser in \cref{conj:mixed} is pure.  Gil-Medrano's theorem proves it for $(N,k)=(2,2)$, where every equatorial projective plane has mixed type $(1,1)$ \cite{GilMedrano2016}.

There are several plausible routes to \cref{conj:mixed}.

\begin{enumerate}
 \item \emph{A $U(N)$-Crofton formula.}  The round proof averages intersections over the full orthogonal group.  The Berger metric only has $U(N)$ symmetry, and the relevant real Grassmannian decomposes into K\"ahler-angle orbits.  A sharp kinematic formula weighted by those orbits may yield the sharp constants required by \cref{conj:mixed}.  Hermitian integral geometry and Tasaki's multiple-angle formalism are natural inputs.
 \item \emph{A calibration-type inequality.}  The candidate tangent planes have a fixed complex/totally-real splitting.  A differential form or convex family of forms that detects both components could give a pointwise lower bound sharper than the bounds in \cref{lem:pointwise-jacobian}.
 \item \emph{Hopf slicing.}  One may separate vertical and horizontal Jacobians, slice by Hopf fibres, and combine a coarea inequality with projective minimisation on the base $\CP^{N-1}$.  The obstruction is that a general mod-two cycle need not be adapted to the fibration.
 \item \emph{Averaging over the candidate orbit.}  The critical orbit in \cref{prop:critical-orbits} carries a natural invariant measure.  An integral-geometric identity involving incidence with this orbit could play the role of the equator averaging formula used in dimension three; compare \cite{AmbrozioMontezuma2020,Ambrozio2026}.
\end{enumerate}

We are not aware of how to prove \cref{conj:mixed}, however we note the following reduction.
The linear results of the present paper fix both the candidate and the sharp constant in every dimension.
The proof of \cref{conj:mixed} therefore reduces to establishing a single global inequality with an already determined equality geometry. 

\section{Further remarks}\label{sec:further}

\subsection{Minimality for one Berger parameter}

The mean-curvature formula has an immediate rigidity consequence.  If a linear equator is minimal for one non-round Berger parameter, then it is minimal for every Berger parameter.

\begin{corollary}\label{cor:all-parameters}
Let $V\subset\C^N$.  The following are equivalent:
\begin{enumerate}
 \item $M_V$ is minimal for some $\mu\ne1$;
 \item $M_V$ is minimal for every $\mu>0$;
 \item $V$ is a linear CR subspace.
\end{enumerate}
\end{corollary}

\begin{proof}
For $\mu\ne1$, the vanishing condition in \cref{prop:mean-curvature} is independent of $\mu$.  At $\mu=1$ every linear equator is totally geodesic.
\end{proof}

This is the higher-dimensional linear analogue of the phenomenon observed for distinguished minimal surfaces in Berger three-spheres \cite{Torralbo2012}.

\subsection{Endpoint asymptotics}

For a mixed CR type $(c,r)$, the beta-integral in \cref{prop:hypergeometric-volume} also records the collapse and stretching asymptotics.  If $c,r>0$, then
\[
 \lim_{\mu\downarrow0}
 \frac{\Vol_{g_\mu}(\RP(\C^c\oplus\R^r))}{v_{d-1}}
 =\frac{B(c,(r+1)/2)}{B(c,r/2)},
\]
and
\[
 \frac{\Vol_{g_\mu}(\RP(\C^c\oplus\R^r))}{v_{d-1}}
 =\sqrt\mu\,
 \frac{B(c+1/2,r/2)}{B(c,r/2)}+o(\sqrt\mu)
 \qquad(\mu\to\infty).
\]
The resulting beta-function constants may be useful in testing any proposed mixed-regime inequality.

\subsection{What remains beyond the linear problem}

Even if \cref{conj:mixed} is established, several natural variational questions remain.  The mixed CR equators are not totally geodesic, so their Jacobi operators contain genuine contributions from the second fundamental form.  Determining their Morse index and nullity, and deciding whether the critical orbits in \cref{prop:critical-orbits} are Morse--Bott for the unrestricted area functional, would connect the present calculations with the index theory of Berger submanifolds \cite{TorralboUrbano2022} and the programme proposed in \cite{Ambrozio2026}.  A quantitative version of \cref{conj:mixed}, controlling distance from the $U(N)$-orbit by the volume excess, would be especially useful.

\end{document}